\documentclass[11pt,a4paper]{article}

\usepackage[margin=2.2cm]{geometry}

\usepackage[latin1]{inputenc}
\usepackage{amsmath,amssymb,amsthm}

\usepackage{bbm}
\usepackage{enumerate}

\usepackage[colorlinks=true,citecolor=black,linkcolor=black,urlcolor=blue]{hyperref}
\usepackage{hyphenat}

\usepackage{overpic}
\usepackage{graphicx}
\usepackage{wrapfig}
\usepackage{rotating}

\usepackage{marginnote}

\usepackage[numbers]{natbib}
\usepackage[font=small,labelfont=bf]{caption}

\usepackage{wrapfig}

\newtheorem{theorem}{Theorem}
\newtheorem*{theorem*}{Theorem}
\newtheorem*{lemma*}{Lemma}
\newtheorem{lemma}[theorem]{Lemma}

\newtheorem{proposition}[theorem]{Proposition}

\newtheorem*{observation*}{Observation}
\newtheorem*{claim*}{Claim}
\newtheorem{claim}{Claim}

\newcommand{\E}{\mathbb{E}}
\newcommand{\Pb}{\mathbb{P}}

\newcommand{\Z}{\mathbb{Z}}
\newcommand{\N}{\mathbb{N}}

\newcommand{\Gnp}{\mathcal{G}(n,p)}
\newcommand{\Gnm}{\mathcal{G}(n,m)}
\newcommand{\Gnmj}{\mathcal{G}(n_j,m_j)}

\newcommand{\mc}{\mathcal}
\newcommand{\parenth}[1]{\left(#1\right)}

\newcommand{\ceilnk}{{\left\lceil \frac{n}{k}\right\rceil}}
\newcommand{\floornk}{{\left\lfloor \frac{n}{k}\right\rfloor}}
\newcommand{\bfr}{\mathbf{r}}

 \title{Sharp concentration of the equitable chromatic number of 
dense random graphs}

\renewcommand{\le}{\leqslant}
\renewcommand{\ge}{\geqslant}
\renewcommand{\epsilon}{\varepsilon}

\renewcommand{\marginnote}[2][]{}

\author{Annika Heckel
\thanks{Mathematical Institute, University of Oxford,
Andrew Wiles Building, Woodstock Road, Oxford OX2~6GG, UK. E-mail: 
\texttt{heckel@maths.ox.ac.uk}. This research was partially funded by ERC Grant 676632.
}
}
\date{\today}
\begin{document}
 \maketitle
\begin{abstract}
An \emph{equitable colouring} of a graph $G$ is a colouring of the vertices of $G$ so that no two adjacent vertices are coloured the same and, additionally, the colour class sizes differ by at most $1$. The \emph{equitable chromatic number} $\chi_=(G)$ is the minimum number of colours required for this. We study the equitable chromatic number of the dense random graph $\Gnm$, where $m = \left\lfloor p {n \choose 2} \right \rfloor $ and $0<p< 0.86$ is constant. It is a well-known question of Bollob\'as \cite{bollobas:concentrationfixed} whether for $p=1/2$ there is a function $f(n) \rightarrow \infty$ so that for any sequence of intervals of length $f(n)$, the normal chromatic number of $\Gnm$ lies outside the intervals with probability at least $1/2$ if $n$ is large enough. Bollob\'as proposes that this is likely to hold for $f(n) = \log n$. We show that for the \emph{equitable} chromatic number, the answer to the analogous question is negative. In fact, there is a subsequence $(n_j)_{j \in \N}$ of the integers where $\chi_=(\Gnmj) = n/j$ with high probability, i.e., $\chi_=(\Gnmj)$ is concentrated on exactly one explicitly known value. This constitutes surprisingly narrow concentration since in this range the equitable chromatic number, like the normal chromatic number, is rather large in absolute value, namely asymptotically equal to $n / (2\log_b n)$ where $b=1/(1-p)$. 
 \end{abstract}

\section{Introduction}

A \emph{(proper) colouring} of a graph $G$ is an assignment of colours to the vertices so that no two adjacent vertices are coloured the same. The \emph{chromatic number} $\chi(G)$ is  the minimum number of colours required for this.

A common application of the chromatic number are scheduling problems. Suppose that a number of events need to be scheduled, but some pairs of events are in conflict with each other and may not be scheduled at the same time. We can model this problem as a graph where the vertices represent events and the edges represent conflicts. The task is now to group the vertices or events into conflict-free timeslots, which is exactly what is achieved by a proper colouring. The chromatic number is minimum number of timeslots where this is possible.

In practice, it is often also desirable to have a similar number of vertices in each colour class (defined as the set of vertices of one particular colour). In the example above, the size of a colour class corresponds to the number of events which are scheduled to take place at the same time. To optimise the use of available resources, such as the number of rooms in a building, it often makes sense to keep this number as equal as possible across different colour classes.

This additional constraint is reflected in the notion of an \emph{equitable colouring}: an equitable colouring is defined like a normal colouring, but with the additional stipulation that all colour classes are as equal in size as possible. Since the number of colours does not necessarily divide the number of vertices, this means that the colour class sizes may differ by at most $1$. The least number of colours required for an equitable colouring of $G$ is called the \emph{equitable chromatic number $\chi_=(G)$}.

It should be noted that if $G$ can be coloured equitably with $k$ colours, there is not necessarily an equitable colouring with $k+1$ colours. We define the \emph{equitable chromatic threshold} $\chi^*(G)$ as the least number $k$ such that for all $l\ge k$, $G$ allows an equitable colouring with exactly $l$ colours. Note that for any graph $G$,
\[1\le \chi(G) \le \chi_=(G) \le \chi_=^*(G) \le n.\]

While equitable colourings are a natural notion in view of possible applications, they have also been studied extensively in their own right. The famous Hajnal-Szemer{\'e}di Theorem \cite{hajnal1970proof} states that if $\Delta(G)$ is the maximum degree of a graph $G$, then
\[
 \chi_=^*(G) \le \Delta(G)+1.
\]
In this paper, we will study the equitable chromatic number of dense random graphs. For $p \in [0,1]$, we denote by $G \sim \Gnp$ the \emph{binomial random graph} with $n$ labelled vertices where each of the $\binom{n}{2}$ possible edges is present independently with probability $p$. Given integers $n$ and $m \in \{0,\dots, N\}$, where $N={n \choose 2}$, the \emph{uniform random graph} $G \sim \Gnm$ is chosen uniformly at random from all graphs with $n$ labelled vertices and  exactly $m$ edges.

Determining the {chromatic number} of a random graph is a classic challenge in random graph theory and was raised in one of the earliest papers by Erd\H{o}s and R{\'e}nyi \cite{erdos1960evolution}. In a landmark contribution, Bollob{\'a}s first determined the asymptotic value of the chromatic number of the dense random graph $G \sim \mc{G}(n,p)$ where $p$ is constant  \cite{bollobas1988chromatic}. Using martingale concentration inequalities, he proved that whp\footnote{We say that an event $E=E(n)$ holds \emph{with high probability} (whp) if $\lim_{n \rightarrow \infty} \Pb(E) =1$.},
\[
 \chi(G) = (1+o(1)) \frac{n}{2 \log_b n},
\]
where $b=1/(1-p)$. This result was sharpened several times \cite{mcdiarmid1989method,mcdiarmid:chromatic,fountoulakis2010t,panagiotou2009note, heckel:chromaticinpreparation}. 

For equitable colourings, Krivelevich and Patk\'os \cite{krivelevich2009equitable} proved, amongst other things, that for $G \sim \Gnp$ with $n^{-1/5+\epsilon} \le p \le 0.99$, whp 
\[\chi_= (G) \sim \chi(G).\]
They also showed that if $p \le 0.99$ and $\log \log n \ll \log(np)$, then 
\[ \chi^*_= (G) \le (2+o(1)) \chi(G).\]
Rombach and Scott obtain further results in their forthcoming paper \cite{rombachequitable}. 

While the above results deal with the likely \emph{values} of $\chi(\Gnp)$, $\chi_=(\Gnp)$ and $\chi_=^*(\Gnp)$, the question of the \emph{concentration of the chromatic number} is similarly well studied. Shamir and Spencer \cite{shamir1987sharp} showed that for any function $p=p(n)$, $\chi(\Gnp)$ is whp concentrated on an interval of length about $\sqrt{n}$. If $p$ tends to $1$ sufficiently quickly, for example if $p=1-1/(10n)$, this result is asymptotically optimal 
 (see \cite{alon1997concentration}). In contrast, for the equitable chromatic number, no general concentration results are known. This is because the martingale concentration arguments used in \cite{shamir1987sharp} do not apply to equitable colourings.

For sufficiently sparse random graphs, much sharper concentration of the chromatic number is known to hold: 
Alon and Krivelevich \cite{alon1997concentration} proved that if $p=n^{-1/2+\epsilon}$ with $\epsilon>0$, then the chromatic number is in fact concentrated on two values, which is generally best possible. For the wide range of values of $p$ which are larger than $n^{-1/2+\epsilon}$ but not close to $1$ --- in particular for $p=1/2$ --- the question of concentration of $\chi(\Gnp)$ is still wide open. Alon and Krivelevich speculate in \cite{alon1997concentration} that the answer for $\mc{G}(n, 1/2)$ might actually be different for $n$  in different subsequences of the integers.

One very simple reason why extremely sharp concentration of $\chi_=(\Gnp)$ would be surprising is that the chromatic number is quite large in absolute value, namely of order $\Theta\parenth{\frac{n}{\log n}}$. In \cite{alon1997concentration}, Alon and Krivelevich give an argument showing non-concentration of any graph parameter which changes considerably as $p$ increases from $\epsilon$ to $1-\epsilon$. This argument only works for graph parameters of order at least $n \log n$.

Since the number of edges of $\Gnp$ with constant $p$ is of order $n^2$ and has variance of order $n$, a large part of the variance of $\chi(\Gnp)$ may simply be due to variations in the number of edges. Therefore, to study the concentration of the chromatic number, it makes sense to focus our attention on the random graph $\Gnm$  where $m = \left \lfloor p N \right \rfloor$ is fixed.

In \cite{bollobas:concentrationfixed}, Bollob{\'a}s asks whether there is a function $f(n) \rightarrow \infty$ such that for any sequence of intervals $I_n$ where $I_n$ is of length $f(n)$, $\chi(\mc{G}(n, \left \lfloor\frac{1}{2}N\right \rfloor))$ does not lie in any interval of length $f(n)$ with probability at least $1/2$, and conjectures that $f(n)= \log n$, and maybe even $f(n) = n^\epsilon$ could do. 

In this paper, we will show that the corresponding statement about the \emph{equitable chromatic number} does not hold. In fact, we will show that for $0<p<0.86$ constant, there is a subsequence $(n_j)_{j \in \N}$ of the integers such that the equitable chromatic number of $\Gnmj$ with $m_j = \left \lfloor p {n_j \choose 2}\right \rfloor$ is concentrated on exactly one explicitly known value.

\begin{theorem}
 \label{equitableconcentration}
Let $0<p <1-1/e^2 \approx 0.86$ be constant. There exists a strictly increasing sequence of integers $(n_j)_{j\ge 1}$ and $j_0 \ge 1$ such that
\begin{enumerate}[a)]
 \item \label{part1thm} for all $j\ge j_0$, $j | n_j$,
 \item \label{part2thm} letting $b=\frac{1}{1-p}$ and $\gamma_j= 2 \log_b n_j-2 \log_b \log_b n_j -2\log_b 2$,
\[
 \gamma_j=j+o(1) \text{ as }j \to \infty \text{, and}
\]
 \item \label{part3thm} letting $G \sim \mc{G} (n_j, m_j )$ with $m_j = \left \lfloor p  {n_j \choose 2}\right \rfloor$, with high probability as $j \to \infty$,
\[
 \chi_=\parenth{G}= \frac{n_j}{j}.
\]
\end{enumerate}
\end{theorem}
In other words, we can pick a subsequence $(n_j)_{j\ge 1}$ of the integers so that whp as $j \to \infty$, the equitable chromatic number of $G \sim \mc{G} (n_j, m_j )$ with $m_j = \left \lfloor p  {n_j \choose 2}\right \rfloor$ is exactly~$\frac{n_j}{j}$. This concentration is perhaps surprisingly sharp because, like the normal chromatic number, the equitable chromatic number of these dense random graphs is of order $\Theta(n/\log n)$. 

We will prove Theorem \ref{equitableconcentration} in Sections \ref{kufd}--\ref{sectionsecondgnm}. The proof is based on a very accurate calculation of the second moment of the number of equitable $k$-colourings, and relies on choosing the sequence $(n_j)_{j\ge 1}$ in such a way that just as the expected number of equitable colourings starts tending to infinity, all the colour classes have exactly the same size. We will use several lemmas from \cite{heckel:chromaticinpreparation}, where the second moment method was recently used to obtain the currently best upper bound for the normal chromatic number of $\Gnp$ where $p$ is constant.

\section{Outline and notation} \label{kufd}
From now on, fix $p<1-1/e^2$. For $n \in \N$, let $m(n)= \left \lfloor p {n \choose 2} \right \rfloor$ and $G \sim \mathcal{G}(n, m(n))$.

 For two functions $f=f(n)$, $g=g(n)$, we say that $f$ is asymptotically at most $g$, denoted by $f \lesssim g$, if $f(n) \le (1+o(1))g(n)$ as $n \rightarrow \infty$. Analogously, $f \gtrsim g$ means that $f(n) \ge (1+o(1))g(n)$. We write $f = O(g)$ if there are constants $C$ and $n_0$ such that $|f(n)| \le Cg(n)$ for all $n\ge n_0$. Furthermore, we say that $f = \Theta(g)$ if $f = O(g)$ and $g= O(f)$.

For $k \in \N$, we call an ordered partition of $n$ vertices into $k$ parts is called an \emph{ordered $k$-equipartition} if all $k$ parts have size $\ceilnk$ or $\floornk$ and decrease in size (so the parts of size $\ceilnk$ come first, followed by the parts of size $\floornk$). Denote by $X_{n,k}$ the number of ordered $k$-equipartitions of $G$ which induce valid colourings.

We will start with a straightforward analysis of the first moment of $X_{n,k}$ in Section \ref{sectionfirstgnm}. Next, in Section \ref{sectionchoice}, we will show that there is a strictly increasing sequence $(n_j)_{j\ge 1}$ which fulfils parts \ref{part1thm}) and \ref{part2thm}) of Theorem \ref{equitableconcentration}, and furthermore, letting $k_j = n_j/j$, 
\begin{align}
 \E[X_{n_j, k_j}] \to \infty \text{ as }j \to \infty. \nonumber 
\end{align}
We will use the first moment method to show that whp $G$ has no equitable $k$-colouring for any $k<k_j$, so whp
 \[
 \chi_=\parenth{G} \ge k_j= \frac{n_j}{j}.
\]
We will prove the matching upper bound through the second moment method. By the Paley-Zygmund Inequality, for any nonnegative integer random variable $Z$,
\[
 \Pb(Z>0) \ge \frac{\E[Z]^2}{\E[Z^2]}.
\]
Hence, to show that whp an equitable $k_j$-colouring exists, it suffices to show that as $j \rightarrow \infty$,
\begin{equation*}
 \E[X_{n_j, k_j}^2] /  \E[X_{n_j, k_j}]^2 \le 1+o(1).
\end{equation*}
This will be shown in Section~\ref{sectionsecondgnm}, using a number of lemmas from \cite{heckel:chromaticinpreparation} where the corresponding ratio was bounded for colourings of $\Gnp$ which could have two different colour class sizes. In that context, it was sufficient to bound the corresponding expression by $\exp\left(\frac{n}{\log^7 n}\right)$. We need much more accurate calculations in this paper to obtain the bound $1+o(1)$, which only holds in $\Gnm$ and if all colour classes are of exactly the same size.

\section{The first moment} \label{sectionfirstgnm}
We start by analysing the first moment of $X_{n,k}$, and prove a number of technical lemmas which will be needed to choose the sequence $(n_j)_{j \in \N}$. Let 
\begin{equation}\gamma=\gamma(n)= 2 \log_b n-2 \log_b \log_b n-2\log_b 2.\label{defgamma}\end{equation}
We will only consider $k$-colourings where
\[
 k = \frac{n}{\gamma+O(1)}. 
\]
Let $q=1-p$, $N= {n \choose 2}$, and $\delta=\delta_{n,k}= \frac{n}{k}-\left \lfloor \frac{n}{k} \right \rfloor$. If $k$ does not divide $n$, an ordered $k$-equipartition of $n$ vertices consists of $k_{\text{L}}= k_{\text{L}}(n)=\delta_{n,k}k$ larger parts of size $\left \lceil \frac{n}{k} \right \rceil$, followed by $k_{\text{S}}= k_{\text{S}}(n)=(1-\delta_{n,k})k$ smaller parts of size $\left \lfloor \frac{n}{k} \right \rfloor$. If $k$ divides $n$, then all $u= k_{\text{S}}(n)$ parts are of size exactly $\frac{n}{k}=\left \lfloor \frac{n}{k} \right \rfloor$. By Stirling's formula $n! \sim \sqrt{2\pi n} n^n / e^n$, the total number of $k$-equipartitions is
\begin{equation}
\label{Pudefinition}
 P_{n,k} = \frac{n!}{\left\lceil\frac{n}{k}\right\rceil!^{k_{\text{L}}}\left\lfloor\frac{n}{k}\right\rfloor!^{k_{\text{S}}}}= k^n \exp(o(n)).
\end{equation}
A $k$-equipartition induces a valid equitable colouring if and only if exactly 
\begin{equation}
 f=f_{n,k} = k_{\text{L}} {\lceil n/k \rceil \choose 2}+k_{\text{S}}  {\lfloor n/k \rfloor \choose 2} 
= \frac{n\left(\frac{n}{k}-1\right)}{2}+\frac{\delta_{n,k}(1-\delta_{n,k})}{2}k \sim n \log_b n \label{defoffk}
\end{equation}
\emph{forbidden edges} within the parts of the partition are not present in $G$. Let \[
                                                                                                                                 \epsilon=Np-m \in [0,1].
                                                                                                                                \]
Since $f=f_{n,k} \sim n \log_b n$ and using Stirling's formula, the probability that a given $k$-equipartition induces a valid colouring is
\begin{align*}
\frac{{N-f \choose m}}{{N \choose m}} =\frac{(N-f)!(qN+\epsilon)!}{N!(qN-f+\epsilon)!} 
\sim \frac{(N-f)^{N-f}(qN+\epsilon)^{qN+\epsilon}}{N^N(qN-f+\epsilon)^{qN-f+\epsilon}} 
= q^{f}\frac{(1-\frac{f}{N})^{N-f}(1+\frac{\epsilon}{qN})^{qN+\epsilon}}{(1-\frac{f}{qN}+\frac{\epsilon}{qN})^{qN-f+\epsilon}}  .
\end{align*}
Using $\log (1+x) = x - \frac{x^2}{2}+O(x^3)$ for $x \to 0$ and that $\epsilon \in (0,1)$ and $f = \Theta(n \log n)$,
\begin{align*}
 \log \frac{{N-f \choose m}}{{N \choose m}} &= f \log q -\frac{f^2p}{2qN}+O\parenth{\frac{f^3}{N^2}}.
\end{align*}
Since $\frac{f^3}{N^2} =o(1)$, the \emph{expected number of ordered $k$-equipartitions which induce a valid colouring} is
\begin{align}
\mu_{n,k} :=& \E[X_{n,k}] = P_{n,k} \frac{{N-f_{n,k} \choose m}}{{N \choose m}} \sim P_{n,k} q^{f_{n,k}} \exp\left(-\frac{f_{n,k}^2p}{2qN}\right). \label{firstmomentgnmordered}
\end{align}
The expected number of \emph{unordered} equitable partitions which induce valid colourings is
\begin{equation}
 \bar{\mu}_{n,k} = \frac{\mu_{n,k}}{k_\text{L}!k_{\text{S}}!}  \sim  \frac{P_{n,k} q^{f_{n,k}}}{k_\text{L}!k_\text{S}!}  \exp\left(-\frac{f_{n,k}^2p}{2qN}\right).\label{firstmomentgnmunordered}
\end{equation}

If $k$ is not too close to $n/\gamma$, we can approximate $\bar{\mu}_{n,k}$ with the following lemma.
\begin{lemma} \label{lemmawithexponential}
 Given $n$ and $k$ and $x = O(1)$ such that 
\[k = \frac{n}{\gamma+x},\]
then
\[
 \bar{\mu}_{n,k}= b^{-\frac{x}{2}n+o(n)}.
\]
\end{lemma}
\begin{proof}
Since $1 \le {k \choose k_{\text{S}}} \le 2^k$ and $k \sim \frac{n}{2 \log_b n}$, it follows that
\begin{equation}
k_{\text{S}}! k_{\text{L}}!= k!\exp(o(n))  = b^{n/2} \exp\parenth{o(n)}. \label{aksjdhjaksdh1}
\end{equation}
Furthermore, by (\ref{defgamma}) and (\ref{defoffk}),
\begin{equation}
 q^{f_{n,k}} =  b^{-\frac{n\left(\frac{n}{k}-1\right)}{2}} \exp\parenth{o(n)} =  b^{\frac{n}{2}(1-\gamma-x)}= b^{n\frac{1-x}{2}} \parenth{\frac{2 \log_b n}{n}}^n. \label{aksjdhjaksdh2}
\end{equation}
Finally, note that 
\[\exp\left(-\frac{f_{n,k}^2p}{2qN}\right)= \exp \parenth{O \parenth{\log^2 n}} = \exp\parenth{o(n)}.\]
Together with (\ref{Pudefinition}), (\ref{firstmomentgnmunordered}), (\ref{aksjdhjaksdh1}), and (\ref{aksjdhjaksdh2}), this gives
 \[\bar{\mu}_{n,k}= \frac{k^n b^{n\frac{1-x}{2}}}{b^{n/2}} \parenth{\frac{2 \log_b n}{n}}^n \exp\parenth{o(n)}= b^{-\frac{x}{2}n}(1+o(1))^n \exp\parenth{o(n)} = b^{-\frac{x}{2}n+o(n)}.\]
\end{proof}

In Lemma \ref{propsubsequence}, we will pick a sequence $(n_j)_{j \ge 1}$ such that $\bar{\mu}_{n_j,k_j}$ starts tending to infinity just as all parts of a $k_j$-equipartition are of size exactly $j$. For this we will need the following technical lemmas which examine how much ${\mu}_{n,k}$  and $\bar{\mu}_{n,k}$ change if we increase $k$ by $1$.

It should be noted that $\bar{\mu}_{n,k}$ behaves in a slightly irregular way: while ${\mu}_{n,k}$ increases steadily if we increase $k$, as can be seen in Lemma \ref{kchangeby1}, the increases in $\bar{\mu}_{n,k}$ are not as large as one might expect if $n/k$ is close to an integer. This is because the product $ k_\text{L}!k_{\text{S}}!$ is larger when $n/k$ is close to an integer (i.e., if either $ k_\text{L}$ or $k_{\text{S}}$ is close to $k$) than when $n/k$ is sufficiently far away from any integers. However, in Lemma \ref{kchangeby1unordered} we will show that even in the `worst case scenario', $\bar{\mu}_{n,k}$ still increases by a sufficient amount.

\begin{lemma} \label{kchangeby1}
Given $n$ and $k$ such that $k = \frac{n}{\gamma+O(1)}$,
\[
\frac{{\mu}_{n,k+1}}{{\mu}_{n,k}} \gtrsim b^{\frac{n^2}{2k(k+1)}-\frac{n}{2k}}.
\]
\end{lemma}
\begin{proof}
First note that 
\begin{equation}P_{n,k+1} \ge P_{n,k}.\label{pisgreater}\end{equation}
To see that this is true, note that the product $\left\lceil\frac{n}{k}\right\rceil!^{k_{\text{L}}}\left\lfloor\frac{n}{k}\right\rfloor!^{k_{\text{S}}}$ contains exactly $n$ factors, and increasing $k$ by $1$ can only decrease those $n$ factors. Now let
\[
x= \frac{n}{k}-\frac{n}{k+1} = \frac{n}{k(k+1)}.
\]
If $\floornk=\left\lfloor\frac{n}{k+1}\right\rfloor$, then $\delta_{n,k+1}= \delta_{n,k}-x$. Otherwise, $\floornk=\left\lfloor\frac{n}{k+1}\right\rfloor+1$ and $\delta_{n,k}+(1-\delta_{n,k+1})=x$. In both cases,
\[
 |\delta_{n,k+1}(1-\delta_{n,k+1})-\delta_{n,k}(1-\delta_{n,k})| \le x.
\]
Therefore, by (\ref{defoffk}),
\begin{equation}
 f_{n,k} - f_{n,k+1} \ge \frac{n^2}{2k(k+1)} -\frac{x}{2}(k+1) = \frac{n^2}{2k(k+1)} -\frac{n}{2k}. \label{diffoffk}
\end{equation}
Furthermore,
\[| f_{n,k} - f_{n,k+1} | = \frac{n^2}{2k(k+1)} + O \parenth{\frac{n}{k}}=O \parenth{\log^2 n}, \]
so as $f_{n,k}\sim n \log_b n$,
\[
 \exp\left(-\frac{f_{n,k+1}^2p}{2qN}+\frac{f_{n,k}^2p}{2qN} \right) =  \exp\left(o (1) \right) \sim 1.
\]
Together with (\ref{firstmomentgnmordered}), (\ref{pisgreater}),  and (\ref{diffoffk}), this completes the proof.
\end{proof}

\begin{lemma} \label{kchangeby1unordered}
Given $n$ and $k$ such that $k = \frac{n}{\gamma+O(1)}$,
\[
\frac{\bar{\mu}_{n,k+1}}{\bar{\mu}_{n,k}} \ge\exp\parenth{\Theta \parenth{\log n \log \log n}}.
\]
\end{lemma}
\begin{proof}
Let $k_{\text{L}}' = \delta_{n,k+1}(k+1)$ and $k_{\text{S}}'=(1- \delta_{n,k+1})(k+1)$, then 
\[k_{\text{L}}'+ k_{\text{S}}' =k+1= k_{\text{L}}+ k_{\text{S}}+1.\]
If $k_{\text{L}} > \floornk$, then given a $k$-equipartition of $n$ vertices, we can form a $(k+1)$-equipartition by removing one vertex from $\floornk$ parts of size $\ceilnk$ and forming a new part of size $\floornk$ from the removed vertices. In this case, $k_{\text{L}}'=k_{\text{L}}-\floornk$ and $k_{\text{S}}'= k_{\text{S}}+\floornk +1$, and so
\begin{align}
 \frac{ k_{\text{L}}'!k_{\text{S}}'!}{k_{\text{L}}! k_{\text{S}}!} =  \frac{\prod_{t=1}^{\floornk+1} \parenth{ k_{\text{S}} +t}}{\prod_{t=0}^{\floornk-1} \parenth{ k_{\text{L}}-t}} \le  \frac{ (k+1)^{\floornk+1}}{ \floornk!}. \label{lemma1stcase}
\end{align}
Otherwise, if $k_{\text{L}} \le \floornk$, then starting with a $k$-equipartition, we can form a $(k+1)$-equipartition by removing one vertex from each of the $k_{\text{L}}$ parts of size $\ceilnk$ and from $\floornk-k_{\text{L}}$ parts of size $\floornk$, and forming a new part of size $\floornk$ from the removed vertices. In this case, $k_{\text{S}}'=\floornk-k_{\text{L}}$ and $k_{\text{L}}'=k+1 -\floornk+k_{\text{L}} $. Note that for all integers $1 \le x_1\le x_2\le x_3 \le x_4$ with $x_1+x_4=x_2+x_3$, we have $x_1! x_4! \ge x_2! x_3!$. Therefore, if $k_{\text{S}}'\ge k_{\text{L}}$, then
\begin{align}
   k_{\text{L}}'!k_{\text{S}}'! \le k_{\text{S}}'! \parenth{k-k_{\text{S}}'}!(k+1) \le k_{\text{L}}! k_{\text{S}}! (k+1). \label{lemma2ndcase}
\end{align}
Otherwise, if $k_{\text{S}}'= \floornk-k_{\text{L}} < k_{\text{L}} \le \floornk$, then
\begin{align*}
 \frac{  k_{\text{L}}'!k_{\text{S}}'!}{k_{\text{L}}! k_{\text{S}}!} = \frac{ k_{\text{L}}'! /  k_{\text{S}}!}{k_{\text{L}}! /  k_{\text{S}}'!}  \le   \frac{(k+1)^{k_{\text{L}}'-k_{\text{S}}}}{\parenth{k_{\text{L}} -  k_{\text{S}}'    }!}=  \frac{(k+1)^{1+k_{\text{L}}-k_{\text{S}}'}}{\parenth{k_{\text{L}} -  k_{\text{S}}'    }!} \le  \frac{ (k+1)^{\floornk+1}}{ \floornk!}. 
\end{align*}
Comparing this to (\ref{lemma1stcase}) and (\ref{lemma2ndcase}), we can see that in every case,
\begin{align*}
\frac{ k_{\text{L}}'!k_{\text{S}}'!}{k_{\text{L}}! k_{\text{S}}! } &\le  \frac{ (k+1)^{\floornk+1}}{ \floornk!} \le \frac{ e^{\floornk} (k+1)^{\floornk+1}}{ \floornk^{\floornk}} \le \parenth{\frac{e(k+1)}{\frac{n}{k}-1} }^\floornk (k+1) \\
&\le \parenth{\frac{e(k+1)k}{n-k} }^{n/k} (k+1) .
\end{align*}
Together with Lemma \ref{kchangeby1}, and since $\frac{n}{k}= \gamma+O(1)= \Theta (\log n)$, this gives
\begin{align*}
 \frac{\bar{\mu}_{n,k+1}}{\bar{\mu}_{n,k}} &\gtrsim  b^{\frac{n^2}{2k(k+1)}-\frac{n}{2k}}  \parenth{\frac{n-k}{(k+1)k} }^{n/k} n^{O(1)} = b^{\frac{\gamma n}{2(k+1)}}  \parenth{\frac{n}{k^2} }^{\gamma}  n^{O(1)}\\
&=   \parenth{\frac{n}{2 \log_b n} }^{\frac{n}{k+1}}  \parenth{\frac{n}{k^2} }^{\gamma}    n^{O(1)}  = \parenth{\frac{n^2}{k^2 \log_b n}}^\gamma n^{O(1)} = \parenth{\Theta\parenth{\log n}}^\gamma n^{O(1)} \\
&= \exp\parenth{\Theta \parenth{\log n \log \log n}}.
\end{align*}
\end{proof}

We will also need the following lemma which examines how much ${\mu}_{n,k}$ increases if $n$ is increased by $1$.

\begin{lemma} \label{nchangeby1}
Given $n$ and $k$ such that $k = \frac{n}{\gamma+O(1)}$, 
\[
\frac{{\mu}_{n+1,k}}{{\mu}_{n,k}} = \Theta \parenth{\frac{\log n}{n}}.
\]
\end{lemma}
\begin{proof}
Given a $k$-equipartition of $n$ vertices, adding a vertex to a part of size $\floornk$ yields a $k$-equipartition of $n+1$ vertices, so
\begin{equation}
 f_{n+1,k} = f_{n,k}+\floornk. \label{dyhaiuerh}
\end{equation}
Therefore, since $\floornk = \gamma(n)+O(1)$ and by (\ref{defgamma}),
\begin{equation} \label{kiui1}
 q^{f_{n+1,k}-f_{n,k}}= \Theta(1)  \parenth{\frac{2 \log_b n}{n}}^2  .
\end{equation}
Furthermore, as $f_{n,k} = O(n \log n)$ and from (\ref{dyhaiuerh}),
\begin{equation} \label{kiui2}
 -\frac{f_{n+1,k}^2p}{2qN} + \frac{f_{n,k}^2p}{2qN} = O \parenth{\frac{\log^2 n}{n}} = o(1).
\end{equation}
Finally, note that by (\ref{Pudefinition}),
\begin{equation*}
 \frac{P_{n+1,k}}{P_{n,k}} \sim \frac{n}{2 \log_b n},
\end{equation*}
since the factorial in the numerator is multiplied by $n+1 \sim n$ and the product in the denominator is multiplied by exactly one factor $\floornk+1 \sim \gamma\sim 2 \log_b n$ if $n$ is increased to $n+1$. Together with (\ref{firstmomentgnmordered}),  (\ref{kiui1}) and (\ref{kiui2}), this completes the proof.\end{proof}

\section{Choice of the subsequence} \label{sectionchoice}
We are now ready to choose the sequence $(n_j)_{j\ge 1}$ from Theorem~\ref{equitableconcentration}.

\begin{proposition}
 \label{propsubsequence}
There is a constant $j_0 \in \N$ and a strictly increasing sequence $(n_j)_{j \ge 1}$ such that the following holds for all $j \ge j_0$.
\begin{enumerate}[a)]
 \item $k_j := \frac{n_j}{j} \in \Z$.\label{parta}
\item $\gamma_j=j+o(1)$ as $j \to \infty$, where $\gamma_j = 2 \log_b n_j-2 \log_b \log_b n_j -2 \log_b 2$. \label{partb} 
\item $\bar{\mu}_{n_j, k_j}  \to \infty \text{ as }j \to \infty$. 
\label{partc}
\item Let $G \sim \mc{G}(n_j, m_j)$ with $m_j = \left \lfloor p {n_j \choose 2} \right \rfloor$, then whp as $j \to \infty$, for all $k \le k_j-1$, $G$ has no equitable $k$-colouring.
\label{partd}
\end{enumerate}
\end{proposition}

\begin{proof}
It suffices to show that there is a strictly increasing sequence $(n_j)_{j \ge j_0}$ which fulfils \ref{parta})--\ref{partd}) for some $j_0$: given such a sequence, by \ref{partb}) $n_j$ grows exponentially in $j$, so without loss of generality we can assume that $n_{j_0}\ge j_0$ and let $n_j=j$ for $1 \le j<j_0$. From Lemma \ref{lemmawithexponential}, we can make the following easy observation.
\begin{observation*}
Given integers $n$ and $k$ such that $\left|\frac{n}{k}-\gamma(n)\right| \le 10$, if $\bar{\mu}_{n,k} \ge 1$, then
\begin{itemize}
 \item if $\bar{\mu}_{n,k} \ge 1$, then $ k \ge \frac{n}{ \gamma(n)+o_n(1)}$.
\item if $\bar{\mu}_{n,k} \le n$, then $ k \le \frac{n}{ \gamma(n)+o_n(1)}$.
\end{itemize}
\end{observation*}
For $j \in \N$, let $n_j$ be the smallest multiple of $j$ such that, letting $k_j = n_j/j$,
\begin{enumerate}[(1)]
 \item $\gamma(n_j) \in \left[ j-10, j+10 \right]$
 \item $\bar{\mu}_{n_j, k_j} \ge \log j$.
\end{enumerate}
\begin{claim}
 If $j$ is large enough, then $n_j$ is well-defined.
\end{claim}
\begin{proof}Fix $j$ and consider the sequence $(g_t)_{t \ge j}$ with $g_t = \gamma(tj)$. If $j$ is large enough and $t \ge j$, then
\begin{equation}
 0 < g_{t+1} -g_t \le 3 \parenth{\log_b\parenth{(t+1)j}- \log_b \parenth{tj}} = 3 \log_b \parenth{1+\frac{1}{t}} < \frac{1}{10}. \label{sdkjfhkjsdf}
\end{equation}
Furthermore, if $j$ is large enough, then $g_j \le j$ and $g_t \to \infty$ as $t \to \infty$. Therefore, if $j$ is large enough, there is an integer $t_0 \ge j$ such that 
\[
 g_{t_0} \in \left[j+0.1, j+0.2\right].
\]
By Lemma \ref{lemmawithexponential}, letting $n'=t_0 j$,
\[
 \bar{\mu}_{n', t_0} \ge b^{(0.05+o(1))n'} \ge \log n' \ge \log j .
\]
Therefore, $n'$ is a multiple of $j$ such that the two conditions (1) and (2) from the definition of $n_j$ are fulfilled, so $n_j$ is well-defined.
\end{proof}
The definition of $n_j$ immediately implies \ref{parta}) and \ref{partc}). We now show that \ref{partb}) holds.

\begin{claim}
 As $j \to \infty$, $\,\,\, \gamma_j=j+o(1)$.
\end{claim}
\begin{proof}
By definition,
\[
 \bar{\mu}_{n_j, k_j} \ge \log j \ge 1,
\]
so by our observation, 
\[
 k_j = \frac{n_j}{j} \ge \frac{n_j}{ \gamma(n_j)+o(1)}.
\]
In particular, $j \le \gamma(n_j)+o(1)= \gamma_j+o(1) \le 2 \log_b n_j$ for $j$ large enough (note that $n_j\ge j$ by definition).
By (\ref{sdkjfhkjsdf}), for $j$ large enough,
\[
\gamma(n_j) \ge \gamma(n_j-j) \ge \gamma(n_j)-\frac{1}{10} \ge j-\frac{1}{20},
\]
so in particular
\begin{equation*}
 \gamma(n_j -j) \in [j-10, j+10].
\end{equation*}
By the minimality of $n_j$ in its definition, we must therefore have
\begin{equation} \label{djska}
 \bar{\mu}_{n_j-j, k_j-1} < \log j \le \log \big(2 \log_b n_j\big) \le n_j -j.
\end{equation}
By our observation,
\[
 k_j-1 = \frac{n_j-j}{j} \le \frac{n_j-j}{ \gamma(n_j-j)+o(1)},
\]
so $j \ge \gamma(n_j-j)+o(1)$. By (\ref{sdkjfhkjsdf}),
\[
 \gamma(n_j-j)= \gamma(n_j)+ O \parenth{\log \parenth{1+\frac{1}{(n_j-j)/j}}}=  \gamma(n_j)+o(1).
\]
Therefore, $\gamma(n_j)\le j+o(1)$.\end{proof}
 Since $\gamma(n)$ is strictly increasing for large enough $n$, it follows that $n_j$ is strictly increasing if $j$ is large enough. It only remains to show \ref{partd}).

\begin{claim}If $j$ is large enough, then \label{claim2inlemma}
  \[\bar{\mu}_{n_j, k_j-1} \le \frac{1}{j}.     \]
\end{claim}
\begin{proof}
By (\ref{djska}), 
\begin{equation}\bar{\mu}_{n_j-j, k_j-1} = \frac{{\mu}_{n_j-j, k_j-1}}{(k_j-1)!}< \log j. \label{kfgshjd} 
\end{equation}
Note that by Lemma \ref{nchangeby1},
\begin{equation}\frac{{\mu}_{n_j, k_j-1}}{{\mu}_{n_j-j, k_j-1}} \le \parenth{\Theta\parenth{ \frac{\log n_j }{n_j }}}^j. \label{tfgfd} \end{equation}
Since $k_j=\frac{n_j}{j}$, an equitable partition of $n_j$ vertices into $k_j-1$ parts consists of exactly $j$ larger parts of size $j+1$ and $k_j-1-j$ smaller parts of size $j$. Hence,
\[
\bar{\mu}_{n_j, k_j-1}=\frac{{\mu}_{n_j, k_j-1}}{j!(k_j-1-j)!}.
\]
Together with (\ref{kfgshjd}) and (\ref{tfgfd}) and the facts that $j! \ge j^j/e^j$, $\,j=\Theta(\log n_j)$ and $k_j = \Theta\parenth{n_j/\log n_j}$, this gives
\begin{align*}
 \bar{\mu}_{n_j, k_j-1} &\le \log j \parenth{\Theta\parenth{ \frac{\log n_j }{n_j }}}^j\frac{(k_j-1)!}{j!(k_j-1-j)!} \le  \log j \parenth{\Theta\parenth{ \frac{\log n_j }{n_j }}}^j \frac{(k_j-1)^j}{j!}\\
&\le \log j \parenth{\Theta\parenth{ \frac{k_j \log n_j }{jn_j }}}^j = \log j \parenth{\Theta\parenth{1/j}}^j<1/j
\end{align*}
if $j$ is large enough.
\end{proof}
An easy first moment argument shows that whp, $\chi_=(G) \le \chi(G) \le \frac{n}{\gamma(n)+o(1)}$ for $G \sim \mc{G}(n, m(n))$ (see \cite{panagiotou2009note} where this was proved for $\Gnp$; for $\Gnm$ the proof is very similar). So whp $G\sim \mc{G}(n_j, m_j)$ has no colouring and therefore no equitable colouring with fewer than $\frac{n_j}{\gamma_j+o(1)}$ colours. It only remains to prove that whp $G$ has no equitable colouring with at least $\frac{n_j}{\gamma_j+o(1)}$ colours (where the $o(1)$-term is an arbitrary function that tends to $0$) but at most $k_j-1$ colours (of course $k_j-1$ is of the form $\frac{n_j}{\gamma_j+o(1)}$ as well). By Lemma~\ref{kchangeby1unordered} and Claim~\ref{claim2inlemma}, the expected number of (unordered) partitions which induce such a colouring is
\[
 \sum_{\frac{n}{\gamma_j+o(1)}\le k\le k_j-1} \bar{\mu}_{n_j, k} = O \parenth{ \bar{\mu}_{n_j, k_j-1}  } = O \parenth{1/j} =o(1), 
\]
so whp no equitable colouring with at most $k_j-1$ colours exists.
\end{proof}

\section{The second moment}\label{sectionsecondgnm}

For the proof of Theorem \ref{equitableconcentration}, it remains to show that for the sequence $(n_j)_{j \ge 1}$ from Proposition~\ref{propsubsequence},
\begin{equation}
 \E[X_{n_j, k_j}^2] /  \E[X_{n_j, k_j}]^2 \le 1+o(1) \text{ as }j \to \infty.\label{lkjh}
\end{equation}
We will be able to reuse large parts of the calculations from \cite{heckel:chromaticinpreparation}, but will need to be more accurate here than in \cite{heckel:chromaticinpreparation} where it was sufficient to bound the corresponding ratio by $\exp \left( \frac{n}{\log^8 n}\right)$ rather than $1+o(1)$.

Let $j\ge j_0$, $N_j={n_j \choose 2}$, $m_j = \left \lfloor pN_j \right \rfloor$ and $f_j = f_{n_j, k_j}$ as in (\ref{defoffk}) and $P_j = P_{n_j, k_j}$ as in (\ref{Pudefinition}), and  $G \sim \mc{G}(n_j, m_j)$. To simplify notation, we will omit the indices of $n_j$, $m_j$, $k_j$ and so on when the context is clear.

Note that since $j= \frac{n}{k}$ is an integer, we can simplify the expressions for $P$ and $f$:
\begin{align}
P = \frac{n!}{j!^k}\,\,\, \text{ and } \,\,\,f =  \frac{n (j-1)}{2}. \label{equidefofp}
\end{align}
Note that
\[
 X_{n,k} = \sum_{\pi \text{ ordered $k$-equipartition}} \mathbbm{1}_{\pi \text{ induces a valid colouring}},
\]
so by linearity of the expectation,
\[
 \E[X_{n,k}^2] = \sum _{\pi_1, \pi_2 \text{ ordered $k$-equipartitions}} \Pb \left( \text{both $\pi_1$ and $\pi_2$ induce proper colourings} \right).
\]
The terms in the last sum may vary considerably depending on how similar $\pi_1$ and $\pi_2$ are to each other. To quantify this, we define the \textit{overlap sequence} $\mathbf{r}= \mathbf{r}(\pi_1, \pi_2)=(r_i)_{i=2}^j$ of  $\pi_1$ and $\pi_2$. We denote by $ r_i $ the number of pairs of parts (with the first part in $\pi_1$ and the second part in $\pi_2$) which intersect in exactly $i$ vertices. 

Conversely, given an \textit{overlap sequence} $\mathbf{r}$, let $P_{\mathbf{r}}$ denote the number of ordered pairs $\pi_1$, $\pi_2$ which overlap according to $\mathbf{r}$. We call an intersection of size at least $2$ between two parts an \textit{overlap block}. Let
\begin{align}
v = v(\bfr) =\sum_{i=2}^{j} i r_i\label{recalldefinitions}
\end{align}
denote the number of \emph{vertices involved in the overlap}, and denote the proportion of such vertices in the graph by 
\[
\rho = v/n.\]
If $\pi_1$ and $\pi_2$ overlap according to $\bfr$, then they share exactly
\begin{equation}\label{recalldefinitions1}d = d(\mathbf{r}) = \sum_{i=2}^{j} r_i {i \choose 2} \end{equation}                                                                                                                                                                                                             
forbidden vertices. Therefore, $\pi_1$ and $\pi_2$ with overlap sequence $\bfr$ both induce valid colourings at the same time if and only if exactly $2f-d(\bfr)$ given forbidden edges are not present in $G$, so by (\ref{firstmomentgnmordered}),
\begin{align*}
\E[X_{n,k}^2] &= \sum _{\mathbf{r}} P_{\mathbf{r}}  \cdot \frac{{{N-2f+d(\bfr)} \choose m}}{{N \choose m}} 
= \mu_{n,k}^2 \sum _{\mathbf{r}} \frac{P_{\mathbf{r}}}{P^2} \cdot \frac{{{N-2f + d \choose m}}  {N \choose m}  }{{N-f \choose m}^2}.
\end{align*}
Let
\begin{align}
 Q_\bfr &=  \frac{P_\bfr}{P^2} \nonumber \\
 S_\bfr &=  \frac{{{N-2f + d(\bfr) \choose m}}  {N \choose m}  }{{N-f \choose m}^2}, \label{equidefofqands}
\end{align}
then to prove (\ref{lkjh}), we need to show that as $j \to \infty$,
\begin{equation}
 \sum _{\mathbf{r}} Q_\bfr S_\bfr \le 1+o(1).  \label{goalequitable}
\end{equation}
We will determine the asymptotic value of $S_\bfr$ in Section \ref{sectionsbfr}. To bound the sum (\ref{goalequitable}), we will split up the calculations into three different ranges in Sections \ref{sectionexact}-- \ref{sectionhighoverlapequi}.

In Section \ref{sectionexact}, we will consider all those overlap sequences $\bfr$ where $\rho=v/n \le c$ for some constant $c>0$. Most pairs of $k$-equipartitions belong in this range, and this is also where the main contribution to (\ref{goalequitable}) comes from. While in \cite{heckel:chromaticinpreparation} it was sufficient to bound the contribution from this range to the equivalent of (\ref{goalequitable}) by $\exp\left(  \frac{n}{\log^8 n} \right)$, in this paper we will show how to obtain the sharpest possible bound $1+o(1)$ (which only holds in $\Gnm$). In Section \ref{sectionotherranges}, we shall consider those $\bfr$ where there are at least $v=cn$ vertices involved in the overlap, but there are either many vertices not in the overlap or many small overlap blocks. We will show that the contribution from these overlap sequences $\bfr$ to (\ref{goalequitable}) is $o(1)$, using a simplification of arguments in \cite{heckel:chromaticinpreparation}. Finally, in Section \ref{sectionhighoverlapequi}, we will discuss overlap sequences $\bfr$ which correspond to pairs of partitions which are very similar to each other, and prove that the contribution from this range of $\bfr$ is $O(1/\bar{\mu}_{n,k})$, which tends to $0$ since $\bar{\mu}_{n,k} \rightarrow \infty$ (and this only holds in general if all colour classes are of exactly the same size). This will conclude the proof of Theorem \ref{equitableconcentration}.

\subsection{Asymptotics of $S_\bfr$} \label{sectionsbfr}
\label{asymptoticsofssection}
Consider an overlap sequence $\bfr$. Again letting $\epsilon=Np-m \in [0,1]$, and $d=d(\bfr)$,
\begin{align*}
 S_\bfr &= \frac{(N-2f+d)!\, N!\,  \left(N-m-f\right)!^2}{(N-f)!^2\, (N-m)!\, \left(N-m-2f+d\right)!}= \frac{(N-2f+d)!\, N!\,  \left(qN+\epsilon-f\right)!^2}{(N-f)!^2\, (qN+\epsilon)!\, \left(qN+\epsilon-2f+d\right)!}.
\end{align*}
Since $d \le f = O(n \log n) = o(N)$, applying Stirling's formula $n! \sim \sqrt{2 \pi n} \, n^n / e^n$ gives
\begin{align*}
 S_\bfr &\sim \frac{(N-2f+d)^{N-2f+d} N^N  \left(qN+\epsilon-f\right)^{2qN+2\epsilon-2f}}{(N-f)^{2N-2f} (qN+\epsilon)^{qN+\epsilon} \left(qN+\epsilon-2f+d\right)^{qN+\epsilon-2f+d}}\\
&= q^{-d} \cdot \frac{\left(1-\frac{2f-d}{N}\right)^{N-2f+d}  \left(1-\frac{f-\epsilon}{qN}\right)^{2qN+2 \epsilon-2f}}{\left(1-\frac{f}{N}\right)^{2N-2f} (1+\frac{\epsilon}{qN})^{qN+\epsilon} \left(1-\frac{2f-d-\epsilon}{qN}\right)^{qN+\epsilon-2f+d}} \\
&\sim  b^{d} \cdot \frac{\left(1-\frac{2f-d}{N}\right)^{N-2f+d}  \left(1-\frac{f-\epsilon}{qN}\right)^{2qN-2f}}{\left(1-\frac{f}{N}\right)^{2N-2f} (1+\frac{\epsilon}{qN})^{qN} \left(1-\frac{2f-d-\epsilon}{qN}\right)^{qN-2f+d}} .
\end{align*}
Using $\log (1+x) = x - \frac{x^2}{2}+O(x^3)$ for $x \to 0$, and as $d\le f = O(n \log n)$, we get
\begin{align*}
 \log S_\bfr &= d \log b +\frac{d^2}{2N}  + \frac{f^2}{N} - \frac{d^2}{2qN}  - \frac{f^2}{qN} - \frac{2df}{N}   + \frac{2df}{qN}  + O\left(\frac{f^3}{N^2}\right)\\
&=d \log b +\frac{-p(d^2+2f^2-4df)}{2qN} +  o(1).
\end{align*}
Therefore,
\begin{align}
S_\bfr \sim b^{d} \exp \left( -\frac{p(d^2+2f^2-4df)}{2qN} \right), \label{asymptoticsofs}
\end{align}
where $d = d(\bfr)$ is given in (\ref{recalldefinitions1}).

\subsection{The typical overlap case} \label{sectionexact}
Let
\begin{align}
 c&=\frac{1}{2}\parenth{1-\frac{\log b}{2}} \in (0,1),\label{defofcequi} \\
\mc{R}_1 &= \left\lbrace \bfr \mid \rho =\rho(\bfr) \le c \right\rbrace \nonumber.
\end{align}
Lemma 9 from \cite{heckel:chromaticinpreparation} gives 
\[
Q_{\bfr} b^d \lesssim 
 \prod_{i=2}^a \left(\frac{1}{r_i!} \left(\frac{e^{\rho i}b^{i \choose 2} k^{2} j!^{2}}{ n^i i!\left( j-i\right)!^{2}}\right)^{r_i} \right)  \exp \left( -\frac{1}{2} \left(\frac{n-v}{k} -1 \right)^2 \right).
\]
(Note that the quantity $x_0$ from \cite{heckel:chromaticinpreparation} is $0$ as $p<1-1/e^2$.) Together with (\ref{asymptoticsofs}), this gives
\[
 Q_{\bfr} S_\bfr \lesssim 
 \prod_{i=2}^a \left(\frac{1}{r_i!} \left(\frac{e^{\rho i}b^{i \choose 2} k^{2} j!^{2}}{ n^i i!\left( j-i\right)!^{2}}\right)^{r_i} \right)  \exp \left( -\frac{1}{2} \left(\frac{n-v}{k} -1 \right)^2  -\frac{p(d^2+2f^2-4df)}{2qN} \right).
\]

Letting
\[
 T_i :=\frac{e^{\rho i}b^{i \choose 2} k^{2}j!^{2}}{ n^i i!\left( j-i\right)!^{2}},
\]
we have
\begin{align}
Q_{\bfr} S_\bfr \lesssim  \exp \left( -\frac{1}{2} \left(\frac{n-v}{k} -1 \right)^2 -\frac{p(d^2+2f^2-4df)}{2qN}\right) \prod_{i=2}^j \frac{\tilde T_i^{r_i}}{r_i!}.  \label{equitablezwischen}
\end{align}
Recalling that $p<1-1/e^2$ and therefore $\log b <2$, let
\[
 \tilde c = \min \parenth{ \frac{1}{2}\parenth{\frac{1}{\log b}-\frac{1}{2}}, \frac{1}{2}}\in(0,1). 
\]

\begin{lemma} \label{lemmaatmostnconstantequitable}
If $j$ (and therefore $n=n_j$) is large enough and $\rho \in \mc{R}_1$, then for all $3\le i \le j$,
\[
T_i \le n^{-\tilde c }.
\]
\end{lemma}

\begin{proof}
 We assume throughout that $j$ is large enough for all bounds to hold. We first check the cases $i=3$ and $i=j$. Note that
\begin{equation*} 
 T_3 \le \frac{e^{3}b^{3} k^{2}j^{6}}{ n^3 3!} = n^{-1+o(1)} \le n^{-\tilde c}.
\end{equation*}
Since $j=\gamma_j+o(1)= 2 \log_b n-2 \log _b \log_b n-2 \log_b 2 +o(1)$,
\[
 b^{{j \choose 2}} = \parenth{\frac{n}{2 \log_b n}}^{j-1} n^{o(1)},
\]
so
with Stirling's formula and as $u=n/j = n^{1+o(1)}$ and $j \sim 2 \log_b n$,
\[
 T_j = \frac{e^{\rho j}b^{j \choose 2} k^{2}j!}{ n^j } = n^{o(1)}\frac{e^{\rho j} nj^j}{\parenth{2 \log_b n}^{j-1}e^j}= n^{1+o(1)}e^{-j+\rho j}.
\]
As $e^{ j} = n^{\frac{2}{\log b}+o(1)}$, and $\rho \le c =\frac{1}{2}\parenth{1-\frac{\log b}{2}} $ since $\bfr \in \mc{R}_1$,  
\begin{equation*}
 T_j \le n^{1-(1-\rho)\frac{2}{\log b}+o(1)} = n^{\frac{1}{2}-\frac{1}{\log b}+o(1)} \le n^{-\tilde c}.
\end{equation*}
Note that
\begin{equation*}
 \frac{ T_{i+1}}{ T_i}= \frac{e^\rho b^i (j-i)^2}{n(i+1)}.\end{equation*}
In particular, for all $i \le 0.8 \log_b n$,
\[
\frac{T_{i+1}}{ T_i } \le n^{-0.2+o(1)}  \le 1, 
\]
so for all $3\le i\le 0.8 \log_b n$, we have $ T_i \le  T_3\le n^{-\tilde c}$. If $i\ge1.2 \log_b n$, then
\[
\frac{ T_{i+1}}{ T_i } \ge n^{0.2+o(1)} \ge 1,
\]
so for all $1.2 \log_b n \le i \le j$, we have $ T_i \le T_j \le n^{-\tilde c}$. For the remaining case $0.8 \log_b < i < 1.2 \log_b n$, note that as $j \le 2 \log_b n$,
\begin{align*}
 T_i \le \frac{e^{ i}b^{\frac{i^2}{2}} n^{2}j^{2i}}{ n^i} \le n^{O(1)}  \frac{n^{0.6 i} j^{2i}}{ n^i} \le n^{O(1)-0.3i} = n^{-\Theta(\log n)} \le n^{-\tilde c}.
\end{align*}
\end{proof}
Let 
\[
 R_3 = \sum_{i=3}^j r_i,
\]
then by Lemma \ref{lemmaatmostnconstantequitable}, (\ref{equitablezwischen}) and the definition (\ref{equidefofp}) of $f$,
\begin{align}
Q_{\bfr} S_\bfr &\lesssim  \exp \left( -\frac{1}{2} \left(\frac{n-v}{k} -1 \right)^2 -\frac{p(d^2+2f^2-4df)}{2qN}\right) \frac{ T_2^{r_2}}{r_2!}n^{-\tilde c R_3} \nonumber \\
&\le \exp \left( -\frac{1}{2} \left(\frac{n-v}{k} -1 \right)^2 -\frac{p(f^2-2df)}{qN}\right) \frac{ T_2^{r_2}}{r_2!}n^{-\tilde c R_3} \nonumber \\
& \sim \exp \left( -\frac{1}{2} \left(\frac{n-v}{k} -1 \right)^2 -\frac{p(j-1)^2}{2q}+\frac{2pd(j-1)}{qn}\right) \frac{ T_2^{r_2}}{r_2!}n^{-\tilde c R_3}. \label{equialmost}
\end{align}
Before summing (\ref{equialmost}) over $\bfr \in \mc{R}_1$, we make a simple observation.
\begin{lemma}\label{easyobservation}
Given $R_3$, there are at most $(2 e \log_b n)^{R_3}$ ways to select $r_3, \dots, r_j$ such that $\sum_{i=3}^j r_i = R_3$.
\end{lemma}
\begin{proof}
Since $j \le 2 \log_b n$, there are at most
\[{{R_3+j-3} \choose R_3} \le \left(\frac{e \left(R_3+j-3\right)}{R_3} \right)^{R_3} \le \left(e\left(1+j-3\right) \right)^{R_3} \le (2e \log_b n)^{R_3}\]
ways to write $R_3$ as an ordered sum of $j-2$ nonnegative summands.
\end{proof}
We first handle the case where either $v$ or $d$ is very large.
\begin{lemma} \label{equitableexceptional}
Let $\mc{R}_1^{\mathrm{ex}}$ be the set of all $\bfr \in \mc{R}_1$ with $v= v(\bfr) \ge \frac{n}{\log^3 n}$ or $d= d(\bfr) \ge \frac{n}{ \log^3 n}$, then
\[
\sum_{\bfr \in \mc{R}_1^{\mathrm{ex}}} Q_{\bfr} S_\bfr = o(1).
\]
\end{lemma}
\begin{proof}
Again we assume throughout that $j$ is large enough for all estimates to be valid. Let $\bfr \in \mc{R}_1^{\mathrm{ex}}$. We first note that as $d \le f = O(n \log n)$, $j = \Theta(\log n)$ and $u= \Theta(n/\log n)$,
\begin{equation}\label{yuyuyuy1}
\exp \left( -\frac{1}{2} \left(\frac{n-v}{k} -1 \right)^2 -\frac{p(j-1)^2}{2q}+\frac{2pd(j-1)}{qn}\right)= \exp\parenth{O(\log^2 n)},
\end{equation}
and 
\begin{equation}\label{yuyuyuy2}
 T_2 \le \frac{e^{2}b^{2 \choose 2} k^{2}j^{4}}{ n^2 2!}= \Theta(\log^2 n).
\end{equation}
Since $v= \sum_{i=2}^j i r_i \le 2 r_2 +2 R_3\log_b n $ and $d= \sum_{i=2}^j {i \choose 2} r_i \le r_2+2 R_3 \log_b^2 n $, if $\bfr \in \mc{R}_1^{\text{ex}}$, then either $r_2 \ge n / \log^6 n$ or $R_3 \ge n / \log^6 n$.

\begin{itemize}
 \item[] \textbf{Case 1: } $r_2 \ge n / \log^6 n$.

 Then from (\ref{equialmost}), (\ref{yuyuyuy1}), and (\ref{yuyuyuy2}), and as $r_2! \ge {r_2}^{r_2}/e^{r_2}$,
\begin{align*}
 Q_{\bfr} S_\bfr &\lesssim \exp\parenth{O(\log^2 n)}\frac{\parenth{\Theta\parenth{\log^2 n}}^{r_2}}{r_2!}n^{-\tilde c R_3}\\
&\le \exp\parenth{O(\log^2 n)} \parenth{\frac{\Theta\parenth{\log^2 n }}{r_2}}^{r_2} n^{-\tilde c R_3}\\
&\lesssim \exp\parenth{O(\log^2 n)} \parenth{\frac{\log^9 n}{n}}^{r_2} n^{-\tilde c R_3}.
\end{align*}
With Lemma \ref{easyobservation}, summing over $r_2$ and $R_3$ gives
\begin{align} 
\sum_{\substack{\bfr \in \mc{R}_1^{\text{ex}} \\ r_2 \ge n / \log^6 n}} Q_{\bfr} S_\bfr  &\le \exp\parenth{O(\log^2 n)}\sum_{r_2 \ge  n / \log^6 n, \,\, R_3} \parenth{ \parenth{\frac{\log^9 n}{n}}^{r_2} \parenth{\frac{2 e \log_b n}{n^{\tilde c}}}^{R_3}} \nonumber \\
&=o(1). \nonumber
\end{align}

\item[] \textbf{Case 2: } $R_3 \ge n / \log^6 n$.

By Lemma  \ref{easyobservation}, (\ref{yuyuyuy1}), and (\ref{yuyuyuy2}),
\begin{align} 
\sum_{\substack{\bfr \in \mc{R}_1^{\text{ex}} \\ R_3 \ge n / \log^6 n}} Q_{\bfr} S_\bfr  &\le \exp\parenth{O(\log^2 n)} \sum_{ \,\, R_3 \ge n / \log^6 n}   \parenth{\frac{2 e \log_b n}{n^{\tilde c}}}^{R_3} \nonumber \\
 &= \exp\parenth{O(\log^2 n)} \parenth{\frac{2 e \log_b n}{n^{\tilde c}}}^{\frac{n}{\log^6 n}}\sum_{ t \ge 0}  \parenth{\frac{2 e \log_b n}{n^{\tilde c}}}^{t} =o(1).\nonumber
\end{align}
\end{itemize}
\end{proof}

We will now sum (\ref{equialmost}) for all $\bfr \in \mc{R}_1 \setminus \mc{R}_1^{\text{ex}}$. If $v < \frac{n}{\log^3 n}$ and $d < \frac{n}{ \log^3 n}$, then 
\begin{align}
\exp \left( -\frac{1}{2} \left(\frac{n-v}{k} -1 \right)^2 -\frac{p(j-1)^2}{2q}+\frac{2pd(j-1)}{qn}\right)\sim \exp \parenth{-\frac{b}{2}(j-1)^2}. \label{exponentialsimplification}
\end{align}
Note that for $\bfr \in \mc{R}_1 \setminus \mc{R}_1^{\text{ex}}$, $\rho=  v/n < \frac{1}{\log^3 n}$, so
\begin{align*}
 T_2 &=\frac{e^{2\rho}b^{2 \choose 2} k^{2}j!^{2}}{ n^2 2!\left( j-2\right)!^{2}}  = e^{2\rho}\frac{b}{2}(j-1)^2\le e^{\frac{2}{\log^3 n}} \,\, \frac{b}{2}(j-1)^2 =:T.
\end{align*}
Therefore, from (\ref{equialmost}) and together with (\ref{exponentialsimplification}) and Lemma \ref{easyobservation},
\begin{align*}
 \sum_{\bfr \in \mc{R}_1 \setminus \mc{R}_1^\text{ex}} Q_\bfr S_\bfr &\lesssim \exp \parenth{-\frac{b}{2}(j-1)^2} \sum_{r_2, R_3 \ge 0} \frac{ T ^{r_2}}{r_2!} \parenth{2 e n^{-\tilde c} \log_b n}^{R_3}\\
& \sim \exp \parenth{-\frac{b}{2}(j-1)^2} \sum_{r_2 \ge 0} \frac{T ^{r_2}}{r_2!} = \exp \parenth{-\frac{b}{2}(j-1)^2+T} \\
&= \exp \parenth{-\frac{b}{2}(j-1)^2 \parenth{1-e^{\frac{2}{\log^3 n}}}} = \exp\parenth{-\frac{b}{2}(j-1)^2 O\parenth{\log^{-3}n}} =1+o(1),
\end{align*}
since $j=O(\log n)$. Together with Lemma \ref{equitableexceptional}, it follows that the contribution from $\mc{R}_1$ to (\ref{goalequitable}) is $1+o(1)$.

\subsection{The intermediate overlap case} \label{sectionotherranges}
Now let $0<c'<1$ be an arbitrary constant and let $ \mc{R}_2^{c'}$ denote the set of all overlap sequences which are not in $\mc{R}_1$ and where there are either at least $c'n$ vertices not in the overlap at all, or at least $c'n$ vertices in `small' overlap blocks of size at most $0.6\gamma$ (the maximum size of any overlap block is $j=\gamma+o(1)$), that is,
\[
 \mc{R}_2^{c'} = \left \lbrace \bfr \mid \rho >c  \wedge \parenth{\sum_{2 \le i \le 0.6\gamma} i r_i \ge {c'} n  \vee \rho \le 1-{c'}}\right \rbrace.
\]
We will show that the contribution to (\ref{goalequitable}) from overlap sequences $\bfr \in \mc{R}_2^{c'}$ is $o(1)$, simplifying arguments from \cite{heckel:chromaticinpreparation}. Fix an arbitrary ordered $k$-equipartition $\pi_1$, and let
\[
\mc{P}_2^{c'} = \left\lbrace \text{ordered $k$-equipartitions $\pi_2$ such that }\bfr (\pi_1,\pi_2) \in \mc{R}_2^{c'} \right \rbrace.
\]
Following \cite{heckel:chromaticinpreparation}, for an overlap sequence $\bfr$, denote by $P'_{\mathbf{r}}$ the number of ordered $k$-equipartitions with overlap $\mathbf{r}$ with $\pi_1$. Then by the definition of $Q_\bfr$,
\begin{align*}
Q_{\mathbf{r}} &=  \frac{P_{\mathbf{r}}}{P^2}= \frac{P'_{\mathbf{r}}}{P}  .
\end{align*}
Using (\ref{asymptoticsofs}) and the fact that $P=k^n \exp (o(n))$, if $n$ is large enough,
\begin{align}
\sum_{\bfr \in \mc{R}_2^{c'}} Q_{\bfr} S_{\bfr} &  = \sum_{\bfr \in \mc{R}_2^{c'}} \frac{P'_\bfr}{P} b^d \exp (o(n))= \sum_{\pi_2 \in \mc{P}_2^{c'}} P^{-1} b^{d(\pi_1,\pi_2)} = \sum_{\pi_2 \in \mc{P}_2^{c'}} k^{-n} b^{d(\pi_1,\pi_2)} \exp(o(n))\label{conversion}
\end{align}
where $d(\pi_1,\pi_2):=d(\bfr)$ if $\bfr$ is the overlap sequence of $\pi_1$ and $\pi_2$.

In Lemmas 13 and 15 in \cite{heckel:chromaticinpreparation}, it was shown that that there are three sets of partitions $\mc{P^\text{I}}$, $\mc{P^\text{II}}$ and $\mc{P^\text{III}}$ (we will not repeat their exact definitions, which are quite technical) such that
\[
 \mc{P}_2^{c'} \subset \mc{P^\text{I}} \cup \mc{P^\text{II}} \cup \mc{P^\text{III}}
\]
and that
 \[\sum_{\pi_2 \in \mc{P}^{\text{I}} \cup  \mc{P}^{\text{II}}} k^{-n} b^{d(\pi_1,\pi_2)} \exp(o( n)) =o(1).
\]
Therefore, to show that the overall contribution of $ \mc{R}_2^{c'}$ to (\ref{goalequitable}) is $o(1)$, we only need to consider the partitions in $ \mc{P^\text{III}}$. Let
\[a=\left \lfloor \gamma \right \rfloor +1.\]
In \cite{heckel:chromaticinpreparation}, given $\pi_1$ and $\pi_2$, two types of vertices in the overlap are distinguished: those $v_1$ vertices which are in parts of size $a$ in $\pi_1$, and those $v_2$ vertices which are in parts of size at most $a-1$ in $\pi_1$, so that $v_1+v_2=v \le n$. Similarly, there are $d_1$ shared forbidden edges in parts of size $a$ in $\pi_1$, and $d_2$ shared forbidden edges in parts of size at most $a-1$ in $\pi_1$, so that $d_1+d_2 = d \le f\le n^2$. Fixing integers $v_1$, $v_2$, $d_1$ and $d_2$ and letting
\begin{align}\label{small}
 \mc{P}'(v_1,v_2,d_1,d_2) = \left \lbrace \pi_2 \in  \mc{P}^{\text{III}} \mid v_i(\pi_1,\pi_2)=v_i, d_i(\pi_1,\pi_2)=d_i, i=1,2\right \rbrace,
\end{align}
Lemma 17 in \cite{heckel:chromaticinpreparation} states that
\begin{align}
\sum_{\pi_2 \in \mc{P}'(v_1,v_2,d_1,d_2)} k^{-n} b^{d(\pi_1,\pi_2)}\le b^{n(1-\rho)\log_b (1-\rho) +\frac{v_1}{2} - \frac{\Delta v}{2}}\exp(o(n)).
\label{yahoo}
\end{align}
In the context of this paper, this last expression can be simplified. Note that since $j= \gamma+o(1)$, $a= \left \lfloor \gamma \right\rfloor +1$ is either $j+1$ or $j$. All parts in our partitions are of size exactly $j$, so in the first case, there are no vertices in parts of size $a$ (which means $v_1=0$, $v_2=v$). In the second case, all vertices are in parts of size $a$ (which means $v_1=v$, $v_2=0$). The second case can only happen if $\Delta = \gamma - \left\lfloor \gamma \right \rfloor=1+o(1)$. So all sets $\mc{P}'(v_1,v_2,d_1,d_2)$ where this is not the case are empty. As $d \le n^2$, there are at most $\exp(o(n))$ choices for the integers $d_1$ and $d_2$. Together with (\ref{small}) and (\ref{yahoo}), (\ref{conversion}) simplifies to
\begin{align*}
\sum_{\bfr \in \mc{R}_2^{c'}} Q_{\bfr} S_\bfr &\le o(1)+ \sum_{cn \le v \le (1-c')n} b^{n(1-\rho)\log_b (1-\rho) }\exp(o(n)) \le o(1)+ n b^{tn}, 
\end{align*}
where $
 t= \min \parenth{(1- c) \log(1- c), c' \log c' } <0$. Therefore,
\begin{align*}
\sum_{\bfr \in \mc{R}_2^{c'}} Q_{\bfr} S_\bfr =o(1).
\end{align*}

\subsection{The high overlap case} \label{sectionhighoverlapequi}
Given $c'\in(0,1)$, let
\[
 \mc{R}_3^{c'} = \left\lbrace \mathbf{r} \mid \rho > 1-{c'}, \sum_{2 \le i \le 0.6\gamma} i r_i \le {c'} n\right\rbrace.
\]
To finish the proof of Theorem \ref{equitableconcentration}, it remains to prove that there is a constant $c'\in(0,1)$ such that the contribution from the overlap sequences $\bfr \in \mc{R}_3^{c'}$ is $o(1)$. This is a direct consequence of the following lemma from \cite{heckel:chromaticinpreparation}.
 \begin{lemma*}[Lemma 22 in \cite{heckel:chromaticinpreparation}] There is a constant $c'>0$ such that
\begin{equation*}   \sum_{\bfr \in \mc{R}_3^{c'}} Q_{\bfr} b^d = O \parenth{\frac{k_{\text{S}}!k_{\text{L}}!}{Pq^f}  {k \choose k_{\text{S}}}}.\end{equation*}
 \end{lemma*}
In our case all colour classes are of the same size, so $k_{\text{L}}=k$ and $k_{\text{S}}=0$, and together with~(\ref{firstmomentgnmunordered}), this gives
\[
 \sum_{\bfr \in \mc{R}_3^{c'}} Q_{\bfr} b^d = O \parenth{\frac{1}{\bar \mu}}\exp\left(-\frac{f^2p}{2qN}\right).
\]
By (\ref{asymptoticsofs}) and since $d \le f$,
\[
\sum_{\bfr \in  \mc{R}_3^{c'}} Q_\bfr S_\bfr = O \parenth{\frac{1}{\bar \mu}} \exp\left( -\frac{p((2f-d)^2-f^2)}{2qN}\right)=O \parenth{\frac{1}{\bar \mu}}.
\]
This expression is $o(1)$ as soon as $\bar{\mu}\to \infty$, which is indeed the case by our choice of $(n_j)_{j \ge 1}$ --- see \ref{partc}) of Proposition \ref{propsubsequence}. This completes the proof of Theorem \ref{equitableconcentration}.
\qed

\section{Outlook}
While Theorem \ref{equitableconcentration} establishes one point concentration of $\chi_=(\Gnm)$ on a subsequence of the integers, we have no general result which holds for all $n$. For the normal chromatic number, Shamir and Spencer \cite{shamir1987sharp} showed concentration on at most about $\sqrt{n}$ integers for every function $p=p(n)$. It would be interesting to obtain a more general concentration result for the equitable chromatic number which holds for all $n$, or indeed for different functions $p(n)$ or $m(n)$.

We also have not addressed the {equitable chromatic threshold} of dense random graphs in Theorem \ref{equitableconcentration}. It may be possible to extend the second moment arguments to equitable $k'$-colourings where $k'\ge k$. However, this would only show that an equitable $k'$-colouring exists whp for one particular value of $k'$, not that equitable $k'$-colourings exist whp for all such values $k'$ \emph{simultaneously}. To prove such a  statement via the second moment method, we would need good bounds on the rate of convergence of $\E[X_{n,k'}^2]/\E[X_{n,k'}]^2 \rightarrow 1$. Nevertheless, it seems like a reasonable conjecture that in the context of Theorem \ref{equitableconcentration}, whp 
\[\chi^*_=(G_{n_j,m_j})=\chi_=(G_{n_j,m_j})=n/j.\]

  \bibliographystyle{plainnat}

\end{document}